\documentclass[a4paper]{article}

\usepackage{amsthm}
\usepackage{amsmath}
\usepackage{amsfonts}
\usepackage{amssymb}
\usepackage{amscd}
\usepackage{bbm}
\usepackage{csquotes}
\usepackage{color}
\usepackage{filecontents}
\usepackage{tikz-cd}

\usepackage{url}
\usepackage{cleveref}
\usepackage{titling}
\usepackage{marginnote}

\newtheorem{theorem}{Theorem}
\newtheorem{proposition}{Proposition}

\newtheorem{lemma}{Lemma}

\theoremstyle{definition}
\newtheorem{definition}{Definition}
\newtheorem{example}{Example}

\DeclareMathOperator{\Hom}{Hom}

\DeclareMathOperator{\Aut}{Aut}

\DeclareMathOperator{\wt}{wt}

\DeclareMathOperator{\Ker}{Ker}

\DeclareMathOperator{\M}{M}

\DeclareMathOperator{\Sym}{Sym}

\newcommand{\Image}[2]{{#1}({#2})}

\newcommand{\dual}[1]{{{#1}^*}}
\newcommand{\ch}[1]{\hat{#1}}
\newcommand{\duale}[1]{{{#1^*}}}
\newcommand{\che}[1]{{\hat{{#1}}}}

\newcommand{\comment}[1]{}

\newcommand{\card}[1]{|#1|}

\renewcommand{\o}{\omega}

\newcommand{\myset}[1]{\{1,\dots,#1\}}

\newcommand{\id}{\mathbbm{1}}

\newcommand{\U}{\mathcal{U}}
\newcommand{\V}{\mathcal{V}}

\usepackage[affil-it]{authblk}
\title{On extendibility of additive code isometries}
\author{Serhii Dyshko
	\thanks{Electronic address: \texttt{dyshko@univ-tln.fr}}}
\affil{Institut de math\'ematiques de Toulon, Universit\'e de Toulon, France}

\begin{document}
	\maketitle

	\begin{abstract}
		For linear codes, the MacWilliams Extension Theorem states that each linear isometry of a linear code extends to a linear isometry of the whole space.
		But, in general, it is not the situation for nonlinear codes.
		In this paper it is proved, that if the length of an additive code is less than some threshold value, then an analogue of the MacWilliams Extension Theorem holds.
		One family of unextendible code isometries for the threshold value of code length is described.
	\end{abstract}

	\section{Introduction}
		
		The main objective of the coding theory is to study the metric structure of a code. Therefore the classification of code isometries is vital for the completeness of the theory.
		
		\comment{Besides the metric, a code can have additional algebraic structure. The most developed are linear codes.} There is a full description of linear code isometries in a Hamming space.
		The famous MacWilliams Extension Theorem claims that each linear isometry of a linear code can be extended to a linear isometry of the full space.
		
		The description of isometries in terms of extendibility is very convenient, because the classification of all isometries of the full space, both linear and nonlinear, have already been done, for example, in \cite{isomconf}.
		
		Unfortunately, in the case where the linearity of a code is not required, the situation is more complicated. There are nonlinear codes with isometries that do not extend to isometries of the whole space.
		
		In general, it is a difficult task to describe codes, which have only extendible isometries.
		Nevertheless, considering some classes of codes, this problem can be solved in particular cases.
		For example, in \cite{aug1}, \cite{heise} and \cite{kov1} authors described several families of nonlinear codes with all isometries extendible. There they also observed various classes of codes that have unextendible isometries.
		Among the studied families there are some subclasses of codes that achieve the Singleton bound (MDS codes, see \cite[p.~20]{macwilliams}), some subclasses of codes with equal distance between codewords (equidistant codes) and some perfect codes (see \cite[Ch.~\S 11]{macwilliams}).
		
		In this paper, we focus our attention on additive codes and the extendibility of additive code isometries. An additive code is a code that forms a group under addition. An additive isometry of an additive code is an isometry that is a group isomorphism.
		The importance of these codes is due to the fact that additive codes with additional requirement of a special kind of self-orthogonality naturally describe quantum stabilizer codes (see \cite{gruska}).
		
		The results presented in the paper are the following. \Cref{theorem-less-then-m} determines the threshold value of the code length for which an analogue of the MacWilliams Extension Theorem for additive codes holds. By providing \Cref{counter-example-existence}, we proved that in general this result cannot be improved by increasing the bound on the code length.
		
\comment{In \cite{exaut} the author described the general ideas for finding additive isometries of additive codes that are automorphisms, i.e., isometry of the code to itself. The results presented in this paper show that for codes of small length all these automorphisms extend to automorphisms of the whole space. Some examples and facts, based on the theory developed in this paper, will be covered in our future works.
In a forthcoming paper we will also study in details extendibility of the additive code isometries of codes of threshold length.}

	\section{Additive codes and additive isometries}
		Let $L^m$ be a Hamming space.
		There is a full description of linear isometries of linear codes in $L^m$. The map $f: L^m \rightarrow L^m$ is called \emph{monomial} if it acts by permutation of coordinates and multiplications of coordinates by nonzero scalars.
		
		\begin{theorem}[MacWilliams Extension Theorem, see \cite{wood}]\label{thm-mac-william}
			Let $C \subseteq L^m$ be a linear code. Each linear isometry of $C$ extends to a monomial map.
		\end{theorem} 
		
\comment{
		The linear structure allows the usage of algebraic methods in finding and description of good codes. 
		Since additive codes are linear in weaker sense ($\mathbb{F}_p$-linear for some prime $p$), some tools of linear algebra can be used.
}

\comment{		When $C$ is a $K$-linear code and $f$ is a $K$-linear map, the image $f(C)$ is a $K$-linear code.}

\comment{
		In \cite{categoryofcodes} there is presented a categorical approach to the notion of linear code isometries. By analogy, we can define the category of $K$-linear codes in $L^m$ (for some fixed $K$). The objects of the category are $K$-linear codes, and morphisms are $K$-linear contracting maps between codes (for every $x \in C$, $\wt(f(x)) \leq \wt(x)$). The categorical \textit{isomorphisms} of $K$-linear codes is a $K$-linear isometry between two codes.	
}

		The MacWilliams Extension Theorem claims that any linear isometry of a linear code can be extended to a linear isometry of the whole space.		
		A general analogue of the MacWilliams Extension Theorem does not exist for nonlinear codes. This means that there exists a nonlinear code and there exists an isometry of this code that does not extend to an isometry of the whole space. In \cite{isomconf} the author gives a full description of the isometries of the ambient space and in \cite{aug1} there is given such an example of unextendible code isometry.
		
		\begin{theorem}[see \cite{isomconf}]\label{thm-code-isomeorphy}
			Let $F$ be a finite set with at least two elements and let $m$ be a positive integer. A map $f : F^m \rightarrow F^m$ is an isometry if and only if there exist a permutation $\pi \in S_m$ and permutations $\sigma_1, \dots, \sigma_m \in \Sym(F)$ such that for any $x = (x_1,\dots, x_m) \in F^m$,
			\begin{equation*}
				f\big( (x_1,\dots, x_m) \big) = \big( \sigma_1(x_{\pi(1)}), \dots, \sigma_m(x_{\pi(m)}) \big) \;.
			\end{equation*}
		\end{theorem}

		\begin{example}[see \cite{aug1}]
			Suppose $F = \{0,1\}$. Two codes in $F^4$
			\begin{equation*}
			C = \{ (0,0,0,0), (1,1,0,0), (1,0,1,0), (0,1,1,0) \}
			\end{equation*}
			and 
			\begin{equation*}
			D = \{ (0,0,0,0), (1,1,0,0), (1,0,1,0), (1,0,0,1)\}
			\end{equation*}
			are isometric, i.e. there exists an isometry $f: C \rightarrow D$. Indeed, in both codes the distance between two different codewords is $2$, thus any bijection $f: C \rightarrow D$ is an isometry.
			For any position, there exists two different codewords in $D$ that have different values in this position.
			But all the codewords in $C$ have equal values on the fourth position.
			According to \Cref{thm-code-isomeorphy}, any isometry between these two codes
			cannot be extended to an isometry of the space $F^4$.
		\end{example}
		
		As we have already noted in the introduction, the studying of the extendibility property for code isometries in general is difficult and only a few families of codes and their isometries have been properly described. 
		In this paper we focus our attention on the extendibility of additive code isometries.
	
		A code in $L^m$ is called \emph{additive} if it is an additive subgroup of $L^m$.
		An \emph{additive isometry} of an additive code $C$ is an isometry that is a group homomorphism. Evidently, a map $f$ is an additive isometry if and only if $f$ preserves the Hamming weight.
		
		Let $K$ be a subfield of $L$. Along with additive codes we will speak about \emph{$K$-linear codes}, i.e. codes that are $K$-linear subspaces of $L^m$. The notions of additive and $K$-linear codes in $L^m$ are in some sense equivalent.
		Any $K$-linear code is additive and, in the other way, any additive code is $\mathbb{F}_p$-linear, where $p$ is the characteristic of $L$. If $K = L$, a $K$-linear code is linear.
	 	Obviously, any $K$-linear isometry is additive and any additive isometry is $\mathbb{F}_p$-linear. 
		
		\begin{example}\label{ex-1}
			Consider two codes $C_1 = \{ (0, 0, 0), (1,1,0), (\o,0,1), (\o^2,1,1) \}$ and $C_2 = \{(0, 0, 0), (0,\o^2,\o), (1,0,1), (1,\o^2,\o^2) \}$ in $\mathbb{F}_4^3$, where $\mathbb{F}_4 = \{ 0, 1, \omega, \omega^2\}$ and $\omega + 1 = \omega^2$. All the codes are $\mathbb{F}_2$-linear.
			Define a map $f: C_1 \rightarrow C_2$ in the following way:
			$f\big((0,0,0)\big) = (0,0,0)$, $f\big((1,1,0)\big) = (0,\o^2,\o)$, $f\big((\o,0,1)\big) = (1,0,1)$ and $f\big((\o^2,1,1)\big) = (1,\o^2,\o^2)$. Evidently, the map $f$ is $\mathbb{F}_2$-linear and it preserves the Hamming weight. Therefore $f$ is an $\mathbb{F}_2$-linear isometry of the $\mathbb{F}_2$-linear code $C_1$ in $\mathbb{F}_4^3$. Both codes $C_1$ and $C_2$ are not $\mathbb{F}_4$-linear.
		\end{example}
		
		Our main objects of study are the $K$-linear isometries of $K$-linear codes. We begin with the description of all $K$-linear isometries of $L^m$.
		The field $L$ can be observed as a finite-dimensional vector space over $K$. By $\Aut_K(L)$ we denote the group of all $K$-linear invertible maps from $L$ to itself.
		\begin{definition}\label{def-general-monomial}
			A map $f: L^m \rightarrow L^m$ is called \emph{$K$-monomial} if there exist a permutation $\pi \in S_m$ and automorphisms $g_1, \dots, g_m \in \Aut_K(L)$ such that for all $u \in L^m$,
			\begin{equation*}
				f(u) = f\big((u_1, u_2, \dots, u_m)\big) = \big(g_1 (u_{\pi(1)}), g_2(u_{\pi(2)}), \dots, g_m( u_{\pi(m)})\big)\;.
			\end{equation*}
		\end{definition}
		
		\begin{proposition}\label{thm-description-of-space-isometries}
			A map $f: L^m \rightarrow L^m$ is $K$-monomial if and only if it is a $K$-linear isometry.
		\end{proposition}
		\begin{proof}
			The only if part is obvious.
			In the other direction, use \Cref{thm-code-isomeorphy}. Since $K$-linear permutations of $L$ are exactly elements of $\Aut_K(L)$, any $K$-linear isometry is a $K$-monomial map.
		\end{proof}
		
		We call a $K$-linear code isometry \emph{extendible} if it is a restriction of a $K$-monomial map on the code. Otherwise, we call it \emph{unextendible}. The following example shows an unextendible additive code isometry.
		
		\begin{example}\label{counter-example-existence}
		Let $K \subset L$ be a pair of finite fields, $m = \card{K} + 1$ and $\o \in L \setminus K$. Consider two $K$-linear codes $C_1 = \langle v_1, v_2 \rangle_K$ and $C_2 = \langle u_1, u_2 \rangle_K$ in $L^m$ with
		\begin{equation*}
		\left(
			\begin{matrix}
				v_1 \\
				v_2
			\end{matrix}\right) =
		\left(
			\begin{matrix}
				1& 1& \dots & 1 & 0\\
				x_1 & x_2 & \dots & x_{\card{K}} & 1
			\end{matrix}\right)
			\xrightarrow{f}
		\left(
			\begin{matrix}
				1& 1& \dots & 1& 0\\
				\o & \o & \dots & \o& 0
			\end{matrix}\right) =
		\left(
			\begin{matrix}
				u_1 \\
				u_2
			\end{matrix}\right)\;,
		\end{equation*}
		where $x_i \in K$ are all different.
		The $K$-linear map $f: C_1 \rightarrow C_2$, defined by $f(v_1) = u_1$ and $f(v_2) = u_2$, is an isometry.
		Indeed, let $\alpha v_1 + \beta v_2$ be an arbitrary element in $C_1 \setminus\{0\}$, where $\alpha, \beta \in K$. If $\beta = 0$, then $\wt(\alpha v_1 + \beta v_2) = m -1$. If $\beta \neq 0$ then the equation $\alpha + \beta x_i = 0$, where $i \in \myset{\card{K}}$, has exactly one solution $x_i = - \alpha \beta^{-1}\in K$ and thus $\wt(\alpha v_1 + \beta v_2) = m -1$. Therefore, all nonzero elements in $C_1$ have the weight equal to $m - 1$. It is easy to see that all nonzero codewords in $C_2$ also have the weight $m - 1$. The map $f$ maps nonzero elements of $C_1$ to nonzero elements of $C_2$ and hence is an isometry.
		At the same time, there is no $K$-monomial map that acts on $C_1$ in the same way as $f$. The last coordinates of all vectors in $C_2$ are always zero, but there is no such all-zero coordinate in $C_1$.
		\end{example}

	\section{Column spaces}
		Let $K$ be a finite field and let $U$ and $L$ be $K$-linear vector spaces over $K$ of dimensions $k$ and $n$ respectively. Fix bases in $U$ and $L$ and let $b_1, \dots, b_n$ be a basis of $L$ over $K$. For simplicity assume that $L$ is a finite field and $K$ is a subfield of $L$.
		
		Denote by $\M_{a \times b}(F)$ the set of all $a \times b$ matrices with the entries from a field $F$. Let $A \in \M_{k \times n}(L)$ be a matrix and let $v \in L^k$ be a column of $A$.
		Suppose $v_1, \dots, v_n \in U$ is the expansion of $v$ in the basis $b_1, \dots, b_n$. This means that $v = \sum_{i = 1}^n b_i v_i$, where the multiplication is component-wise. Define a \emph{column space} $V \subseteq U$ of the vector $v$ as the $K$-linear span $V = \langle v_1, \dots, v_n \rangle_K$. The definition of a column space does not depend on the choice of a basis of $L$ over $K$.
		Call $\V = (V_1, \dots, V_m)$ the \emph{tuple of spaces} of $A$, where $V_i$ denotes the column space of $i$th column of $A$, for $i \in \myset{m}$. 
		
\comment{THIS IS IMPORTANTReally, suppose $b'_1, \dots, b'_n$ is another basis and $g \in \Aut_K(L)$ is such that $g(b'_i) = b_i$, for $i \in \myset{n}$. In the new basis $v = \sum_{i=1}^n b_i'v_i'$ and the column space of $v$ is $V' = \langle v_1',\dots, v_n'\rangle_K$.
		Let $\sigma \in \Hom_K(L,U)$ be such that $\sigma(b_i) = v_i$. Then $V = \sigma(L)$ and $V' = (\sigma g)(L) = \sigma (g(L)) = \sigma(L)= V$.
		}
		
		\begin{example}\label{ex-2}
			Consider the finite field $\mathbb{F}_4 = \{ 0, 1, \omega, \omega^2\}$, where $\omega + 1 = \omega^2$.
			The matrix $A \in \M_{3 \times 3}(\mathbb{F}_4)$,
			\begin{equation*}
				A = \left(\begin{matrix}
					1 & 1& 0\\
					\o & \o & 0 \\
					1 & 0 & 1
				\end{matrix}\right)\;,
			\end{equation*}
			has the following expansion of columns in the $\mathbb{F}_2$-linear basis $1,\o$ of $\mathbb{F}_4$,
			\begin{equation*}
				\left( \begin{matrix} 1 \\ \o \\ 1 \end{matrix} \right) = \left( \begin{matrix} 1 \\ 0 \\ 1 \end{matrix} \right) + \o\left( \begin{matrix} 0 \\ 1 \\ 0 \end{matrix} \right)\;;
				\left( \begin{matrix} 1 \\ \o \\ 0 \end{matrix} \right) = \left( \begin{matrix} 1 \\ 0 \\ 0 \end{matrix} \right) + \o\left( \begin{matrix} 0 \\ 1 \\ 0 \end{matrix} \right)\;;
				\left( \begin{matrix} 0 \\ 0 \\ 1 \end{matrix} \right) = \left( \begin{matrix}  0\\ 0 \\ 1 \end{matrix} \right) + \o\left( \begin{matrix} 0 \\ 0 \\ 0 \end{matrix} \right)\;.
			\end{equation*}
			The column spaces $V_1,V_2,V_3 \subseteq \mathbb{F}_2^3$ are: $V_1 = \langle (1,0,1) , (0,1,0) \rangle_{\mathbb{F}_2}$, $V_2 = \langle (1,0,0) , (0,1,0) \rangle_{\mathbb{F}_2}$ and $V_3 = \langle (0,0,1) \rangle_{\mathbb{F}_2}$.
		\end{example}
		For two vector spaces $U,L$ over a field $K$ we denote by $\Hom_K(U,L)$ the set of all $K$-linear maps from $U$ to $L$. 	
		
		Suppose $\sigma \in \Hom_K (U,L)$.
		There exists a unique matrix $M \in \M_{k \times n}(K)$ such that for all $a \in U$, $\sigma(u) = M^T u$. Define the \emph{dual map} $\dual{\sigma} \in \Hom_K(L,U)$ as $\dual{\sigma}(b) = M b$ for all $b \in L$. Evidently, $\sigma^{**} = \sigma$.
		Let $X$ be another vector space over $K$. Suppose $\sigma_1 \in \Hom_K(U,L)$ and $\sigma_2 \in \Hom_K(L,X)$. Then $(\sigma_2\sigma_1)^* = \sigma_1^* \sigma_2^*$. Note that if $g \in \Aut_K(L)$, then also $g^* \in \Aut_K(L)$.
		
		The matrix $A \in \M_{k \times m}(L)$ naturally defines a map $\lambda \in \Hom_K(U, L^m)$, $\lambda(u) = A^T u$, where $u \in U$. We present $\lambda$ in the form $\lambda = (\lambda_1, \dots, \lambda_m)$, where $\lambda_i(u)$ is the projection of $\lambda(u)$ on $i$th coordinate, $i \in \myset{m}$, $u \in U$. Obviously, $\lambda_i \in \Hom_K(U, L)$, for $i \in \myset{m}$, and it corresponds to the $i$th column of $A$.
		One can see that for all $i \in \myset{m}$, $\duale{\lambda_i}(L) = V_i$, where $\V = (V_1,\dots,V_m)$ is the tuple of spaces of $A$.

		Let $f: \lambda(U) \rightarrow L^m$ be a $K$-linear map. Define a map $\mu = f\lambda \in \Hom_K(U,L^m)$.
		The following diagram is commutative,
		\begin{equation*}
		\begin{tikzcd}
			U \arrow{r}{\mu} \arrow{dr}[swap]{\lambda} & L^m\\
			{} & L^m \arrow[swap]{u}{f}
		\end{tikzcd}
		\end{equation*}
\comment{		
	\begin{CD}
		U @>\mu>> L^m\\
		@AAidA @AAfA\\
		U @>\lambda>> C
	\end{CD}
}
		Suppose $A' \in \M_{k \times m}(L)$ is such that $\mu(u) = A'^T u$ for all $u \in U$. Let $\U = (U_1, \dots, U_m)$ be the tuple of spaces of $A'$. Note that for any $i \in \myset{m}$, $U_i = \mu_i^*(L)$.
		
		Call two tuples of spaces $\U= (U_1,\dots,U_m)$ and $\V=(V_1,\dots,V_m)$ \emph{equivalent} and denote $\U \sim \V$, if there exists a permutation $\pi \in S_m$ such that $V_i = U_{\pi(i)}$ for all $i \in \myset{m}$.
		\begin{lemma}\label{lemma-equal-column-spaces}
			Let $\sigma,\tau \in \Hom_K(U,L)$. There exists $g \in \Aut_K(L)$ such that $\sigma = g\tau$ if and only if $\sigma^*(L) = \tau^*(L)$.
		\end{lemma}
		\begin{proof}
			The $K$-linear spaces $\sigma^*(L)$ and $\tau^*(L)$ are equal if and only if there exists a map $h \in \Aut_K(L)$ such that $\duale{\tau} h = \duale{\sigma}$, or the same, calculating the dual of both maps, there exists a map $g = \duale{h} \in \Aut_K(L)$ such that $g \tau = \sigma$.
		\end{proof}
		\begin{proposition}\label{extendibility-criterium}
			The $K$-linear map $f$ is extendible if and only if the tuples of spaces $\V$ and $\U$ are equivalent.
		\end{proposition}
		\begin{proof}
			The map $f$ is extendible if and only if there exist a permutation $\pi \in S_m$ and maps $g_1,\dots, g_m \in \Aut_K(L)$ such that $\mu_i = g_i\lambda_{\pi(i)}$, for all $i \in \myset{m}$.
			From \Cref{lemma-equal-column-spaces}, the last statement is equivalent to the existence of a permutation $\pi \in S_m$ such that $U_i = V_{\pi(i)}$.
		\end{proof}

	\section{Characters and their applications}
		The proof of the MacWilliams Extension Theorem firstly appeared in the works of MacWilliams and it was later refined by several authors. Namely, in \cite{wood}, Ward and Wood greatly simplified it, using a character theory approach. Generalized analogues of the MacWilliams Extension Theorem for the codes linear over rings and the related properties were discussed in \cite{greferath}, \cite{wood2} and \cite{wood1} where the authors also used the techniques of the character theory.
		
		Recall the notation and basic properties of characters (for more details see \cite[Ch.~18 \S2]{lang}, \cite[Ch.~5 \S4]{macwilliams} and \cite{wood}). For a finite abelian group $G$ let $\hat{G}$ be the set of all homomorphisms from $(G,+)$ to $(\mathbb{C}^*,\times)$, where $\mathbb{C}^*$ is the multiplicative group of complex numbers. With the defined sum of homomorphisms: for $g,h \in \ch{G}, x \in G$, $(g+h)(x) = g(x)h(x)$, the set $\ch{G}$ form an abelian group and is called a \textit{group of characters}.
		It is proved, that the groups $(G,+)$ and $(\ch{G},+)$ are isomorphic (see \cite{wood}).
		
		Let $G$ be a $K$-linear space of dimension $k$. Fix a $K$-linear basis in $G$ and consider the bilinear form $(-,-)_G: G \times G \rightarrow K$, for any $x,y \in G$, $(x,y)_G = \sum_{i=1}^k x_i y_i$. Let $\pi$ be a nontrivial character in $K$. Define a map $\psi_G: G \rightarrow \ch{G}$ as $\psi_G(x)(y) = \chi_x(y) = \pi((x,y)_G)$, where, $x,y \in G$. Define in $\ch{G}$ a multiplication by scalar $(\lambda g)(x) = g(\lambda x)$, where $x \in G$, $g \in \ch{G}$, $\lambda \in K$.
		It is easy to see that $\ch{G}$ is a vector space over $K$ and the map $\psi_G$ is an isomorphism of $K$-linear spaces. \comment{Indeed, $\psi_G$ preserves the addition, for any $x_1, x_2 ,y \in G$, $\psi_G(x_1 + x_2)(y) = \pi( (x_1 + x_2,y)_G ) = \pi((x_1,y)_G )\pi( (x_2,y)_G) = (\psi_G(x_1) + \psi_G(x_2))(y)$ and for any $x,y \in G, \lambda \in K$, $\psi(\lambda x)(y) = \pi((\lambda x, y)_G) = \pi((x, \lambda y)_G) = \psi(x)(\lambda y) = (\lambda\psi(x))(y)$.}
		
		The important property of characters is their linear independence as complex functions. If $\chi_1, \dots, \chi_k \in \ch{G}$ are different characters and $a_1, \dots, a_k \in \mathbb{C}$, then the equality, for all $x \in G$, $\sum_{i = 1}^k a_i \chi_i(x) = 0$, implies that all $a_i = 0$ (see \cite[p.~283]{lang}). 
		
		Also, it is a well known fact that the weight function can be rewritten as a sum of characters (see \cite[p.~143]{macwilliams}). For the weight function $\wt: G \rightarrow \{0,1\}$, that maps $0$ to $0$ and other elements to $1$, the following holds, for all $a \in G$, \begin{equation*}
		\frac{1}{|G|} \sum_{\chi \in \hat{G}} \chi(a) = \frac{1}{|G|} \sum_{b \in G} \chi_b(a) = 1 - \wt(a)\;.
		\end{equation*}
	
		Recall that $U$ and $L$ are vector spaces over $K$.
		Let $\sigma$ be an element in $\Hom_K(U,L)$. Define a map $\ch{\sigma}: \ch{L} \rightarrow \ch{U}$ as $\che{\sigma}(\chi) = \chi\sigma$, for all $\chi \in \ch{L}$. The map $\ch{\sigma}$ is a $K$-linear homomorphism. Indeed, for any $\chi_1, \chi_2 \in \ch{L}$, $ u \in U$, $(\che{\sigma}(\chi_1 + \chi_2))(u) = \chi_1(\sigma(u))\chi_2(\sigma(u)) = (\che{\sigma}(\chi_1) + \che{\sigma}(\chi_2))(u)$ and for any $\chi \in \ch{L}$, $\lambda \in K$, $u \in U$, $\che{\sigma}(\lambda\chi)(u) = \chi(\lambda\sigma(u)) = \chi(\sigma(\lambda u)) = (\lambda\che{\sigma}(\chi))(u)$.
		
		\begin{lemma}\label{lemma-diagrams-and-morphisms}
			For each $\sigma \in \Hom_K(U,L)$ the following diagram is commutative,
			\begin{equation*}
			\begin{CD}
			U @>\psi_U>> \ch{U}\\
			@AA\duale{\sigma}A @AA\che{\sigma}A\\
			L @>\psi_L>> \ch{L}
			\end{CD}
			\end{equation*}
		\end{lemma}
		\begin{proof}
			
\comment{			
			Prove that the map $\sigma \mapsto \che{\sigma}$ is $K$-linear. For any $\sigma_1,\sigma_2 \in \Hom_K(U,L), \chi \in \ch{L}, u \in U$, $\widehat{\sigma_1 + \sigma_2}(\chi) = \chi(\sigma_1 + \sigma_2) = \chi\sigma_1 + \chi\sigma_2 = (\ch{\sigma_1} + \ch{\sigma_2})(\chi)$ and for any $\sigma \in \Hom_K(U,L), \lambda \in K$, $\widehat{\lambda \sigma}(\chi) = \chi(\lambda\sigma) = \chi(\sigma\lambda) = \lambda \ch{\sigma}(\chi)$.
			
			Since the cardinalities of both spaces $\Hom_K(U,L)$ and $\Hom_K(\ch{L}, \ch{U})$ are equal and the map $\sigma \mapsto \che{\sigma}$ is $K$-linear, it is enough to show that $\sigma \mapsto \che{\sigma}$ is injective. Assume that for $\sigma_1, \sigma_2 \in \Hom_K(U,L)$ for all $\chi \in \ch{L}$, $\chi\sigma_1 = \chi\sigma_2$. Equivalently, since characters are homomorphisms, for all $u \in U, \chi \in \ch{L}$ we get $ \chi((\sigma_1 - \sigma_2)(u)) = 1$. For all $u \in U$, $\wt((\sigma_1 - \sigma_2)(u)) =  1 - \frac{1}{\card{L}}\sum_{\chi \in \ch{L}} \chi((\sigma_1 - \sigma_2)(u)) =  0$ that implies $\sigma_1 = \sigma_2$.
			
			We now prove the commutativity of the diagram. }For $b \in L$ calculate $\che{\sigma}(\psi_L(b)) = \chi_b \sigma$ and $\psi_U(\sigma^*(b)) = \chi_{\sigma^*(b)}$. Let matrix $M \in \M_{k \times n}(K)$ be such that $\sigma(u) = M^T u$ and $\sigma^*(b) = M b$ for $u \in U$, $b \in L$. For all $u \in U$, $\chi_b(\sigma(u)) = \chi_b(M^Tu) = \pi((b,M^T u)_L) = \pi(b^T M^T u) = \pi((Mb,u)_U) = \chi_{\sigma^*(b)}(u)$.
		\end{proof}

		Let $X$ be a set and let $Y$ be a subset of $X$. An indicator function is a map $\id_Y: X \rightarrow \{0,1\}$, such that $\id_Y(x) = 1$ if $x \in Y$ and $\id_Y(x) = 0$ otherwise. Recall that for a map $\lambda \in \Hom_K(U,L^m)$ by $\lambda_i$ we denote the projection of $\lambda$ on the $i$th coordinate and $V_i = \lambda_i^*(L)$, where $i \in \myset{m}$. 
		
		\begin{proposition}\label{thm-weight-representation}
			Let $\lambda \in \Hom_K(U,L^m)$. For any $u \in U$ the following equality holds,
			\begin{equation*}
				\wt(\lambda(u)) = m - \sum_{v \in U} \left(\sum_{i=1}^m \frac{1}{\card{V_i}} \id_{V_i}(v)\right) \chi_v(u) \;.
			\end{equation*}
		\end{proposition}
		\begin{proof}
			For any $u \in U$,
			\begin{equation*}
			\begin{split}
			m - \wt\big(\lambda(u)\big) = \sum_{i=1}^m(1 - \wt(\lambda_i(u))) = \sum_{i = 1}^{m} \frac{1}{|L|} \sum_{\chi \in \hat{L}} \chi\big(\lambda_i(u)\big)\\
			=\frac{1}{\card{L}}\sum_{i=1}^m \sum_{\chi \in \ch{L}} \che{\lambda_i}(\chi)(u) =
			\frac{1}{\card{L}}\sum_{i=1}^m\sum_{\pi \in \ch{U}} \left(\card{\Ker \che{\lambda}_i}\id_{\Image{ \che{\lambda_i}}{\ch{L}}}(\pi)\right) \pi(u)
			\\=\sum_{\pi \in \ch{U}} \left(\sum_{i=1}^m \frac{1}{\card{\Image{\che{\lambda_i}}{\ch{L}}}} \id_{\Image{\che{\lambda_i}}{\ch{L}}}(\pi)\right) \pi(u) \;.
			\end{split}
			\end{equation*}
			Substitute $\pi \in \ch{U}$ by $\psi_U(v)$, for $v \in U$. Consider the fact that $\psi_U(v) = \chi_v \in \ch{U}$. \Cref{lemma-diagrams-and-morphisms} implies
			${\psi_U^{-1}\che{\lambda_i}(\ch{L})}=
			\lambda_i^* \psi_L^{-1}(\ch{L}) = \duale{\lambda_i}(L) = V_i$, hence	
			$\id_{\che{\lambda_i}(\ch{L})} \psi_U = \id_{\psi_U^{-1}\che{\lambda_i}(\ch{L})}= \id_{V_i}$ and $\card{\che{\lambda_i}(\ch{L})} = \card{V_i}$, for $i \in \myset{m}$.
		\end{proof}

	\section{The main theorem}
		Let $K \subseteq L$ be a pair of finite fields. We use the representation of the weight function presented in \Cref{thm-weight-representation} to get a description of $K$-linear isometries of $K$-linear codes in $L^m$.
				
		Let $C$ be a $K$-linear code in $L^m$ with some fixed $K$-linear basis. The matrix $A \in \M_{k \times m}(L)$, with the rows equal to the basis vectors of $C$, is called a \textit{generator matrix} of $C$. Let $\V= (V_1,\dots, V_m)$ the tuple of spaces of $A$. 
		Call $\V$ a \emph{tuple of spaces} of $C$. Since a generator matrix of a code is not unique, a tuple of spaces of a code is also not unique.
		
		\begin{proposition}\label{thm-dimension-of-sum}
			Let $C$ be a $K$-linear code and $(V_1, \dots, V_m)$ be a tuple of spaces of $C$. The equality $\dim_K C = \dim_K \left(\sum_{i = 1}^m V_i\right)$ holds.
		\end{proposition}
		\begin{proof}
			Let $A \in \M_{k \times m}(L)$ be a matrix that correspond to the tuple of spaces $(V_1, \dots, V_m)$, i.e. $V_i$ is a column space of the $i$th column $v_i \in L^k$ of $A$, for all $i \in \myset{m}$. Fix a $K$-linear basis $b_1,\dots, b_n$ of $L$ over $K$ and denote by $v_{ij} \in U$ the $j$th term of the expansion of $v_i$ in the basis, for all $i \in \myset{m}, j \in \myset{n}$. Denote $B \in \M_{k \times nm}(K)$ the matrix formed by $nm$ columns $v_{ij}$, $i \in \myset{m}, j \in \myset{n}$.
			The row rank of $B$ equals to the row rank of $A$ and is equal to $\dim_K C$. From the other side, the column rank of $B$ equals to the dimension of the column space of matrix $B$ and is equal to $\dim_K \sum_{i = 1}^m V_i$.
		\end{proof}
		
		Let $f: C \rightarrow L^m$ be a $K$-linear map. Let $U$ be a vector space over $K$ with the dimension equal to $\dim_K C$. Denote by $\lambda$ the map in $\Hom_K(U,L^m)$ defined as $\lambda(u) = A^T u$, for $u \in U$. Since $\lambda(U) = C$, we can define a map $\mu = f\lambda \in \Hom_K(U,L^m)$.
		Let $A' \in \M_{k \times m}(L)$ be such that $\mu(u) = A'^T u$, for all $u \in U$. Denote by $\U = (U_1,\dots,U_m)$ and $\V = (V_1,\dots, V_m)$ tuples of spaces of matrices $A'$ and $A$ correspondingly.

		\begin{proposition}\label{isometry-main-criterium}
			Let $C$ be a $K$-linear code in $L^m$ and $f: C \rightarrow L^m$ be a $K$-linear map. The map $f$ is an isometry if and only if
			\begin{equation}\label{eq-main-counting-space}
				\sum_{i=1}^m \frac{1}{\card{V_i}} \id_{V_i}=
				\sum_{i=1}^m \frac{1}{\card{U_i}} \id_{U_i}\;.
			\end{equation}
		\end{proposition}

		\begin{proof}	
			By definition, a map $f$ is an isometry if for all $x \in C$, $\wt(x) = \wt(f(x))$, or the same for a $K$-linear map $f$, for all $u \in U$, $\wt(\lambda(u)) = \wt(\mu(u))$. Consequently, using \Cref{thm-weight-representation}, $f$ is an isometry if and only if the following equality of functions holds,
			\begin{equation*}
				\sum_{v \in U} \left(\sum_{i=1}^m \frac{1}{\card{V_i}} \id_{V_i}(v)\right) \chi_v = \sum_{v \in U} \left(\sum_{i=1}^m \frac{1}{\card{U_i}} \id_{U_i}(v)\right) \chi_v\;.
			\end{equation*}
			Since different characters in $\hat{U}$ are linearly independent, the coefficients in the equation are equal for each $v \in U$.
		\end{proof}

		\Cref{isometry-main-criterium} shows that the task of description of $K$-linear isometries can be reformulated in terms of solutions of \cref{eq-main-counting-space},
		where $U_1, \dots, U_m$, $V_1, \dots, V_m$ are spaces in $U$, and dimensions of all spaces are bounded by $n$. 
		We call the couple of tuples of spaces $(\U,\V)$ the \emph{solution}, if $\U$ and $\V$ satisfy \cref{eq-main-counting-space}.
\comment{
	We call two couples of tuples $(\U,\V)$ and $(\U',\V')$ equivalent if $\U \sim \U'$, $\V \sim \V'$ or $\U \sim \V'$, $\V \sim \U'$. Obviously, if $(\U,\V)$ is a solution, then any equivalent couples of tuples is also a solution.
}
		
		Evidently, if $\U \sim \V$, then $(\U,\V)$ is a solution. Call a solution $(\U,\V)$ \emph{trivial} if $\U \sim \V$ and \emph{nontrivial} otherwise.
		To illustrate \Cref{extendibility-criterium}, \Cref{isometry-main-criterium} and give an example of a nontrivial solution, we consider the following example observed in \cite{exaut}.
		\begin{example}\label{ex-3}
			Let the field $\mathbb{F}_4 = \{0,1,\o,\o^2\}$ be generated by $\o^2 = \o + 1$.
			Define an $\mathbb{F}_2$-linear map $f: C \rightarrow \mathbb{F}_4^3$ on the generators in the following way: $f\big((1,1,0)\big) = (1,1,0)$, $f\big((\o, \o, 0)\big) =(1,0,1)$ and $f\big((1,0,1)\big) =(\o, \o, 0)$. 
			Consider the following generator matrix $A$ of $C$ and the corresponding generator matrix $A'$ of $f(C)$,
			\begin{equation*}
			A = \left(\begin{matrix}
			1 & 1& 0\\
			\o & \o & 0 \\
			1 & 0 & 1
			\end{matrix}\right) \xrightarrow{f}
			\left(\begin{matrix}
			1 & 1& 0\\
			1 & 0 & 1\\
			\o & \o & 0
			\end{matrix}\right) = A' \;.
			\end{equation*}
			Calculate the tuples of spaces $V_1, V_2, V_3 \subseteq \mathbb{F}_2^3$ and $U_1, U_2, U_3 \subseteq \mathbb{F}_2^3$. The spaces are: $V_1 = \langle (1,0,1) , (0,1,0) \rangle_{\mathbb{F}_2}$, $V_2 = \langle (1,0,0) , (0,1,0) \rangle_{\mathbb{F}_2}$ and $V_3 = \langle (0,0,1) \rangle_{\mathbb{F}_2}$. In the same way, $U_1 = \langle (1,1,0) , (0,0,1) \rangle_{\mathbb{F}_2}$, $U_2 = \langle (1,0,0) , (0,0,1) \rangle_{\mathbb{F}_2}$ and $U_3 = \langle (0,1,0) \rangle_{\mathbb{F}_2}$. 
			The defined spaces $V_1, V_2, V_3$ and $U_1, U_2, U_3$ satisfy the equation,
			\begin{equation*}
			\id_{V_1} + \id_{V_2} + 2 \id_{V_3} = \id_{U_1} +  \id_{U_2} + 2 \id_{U_3}\;,
			\end{equation*}
			and therefore satisfy \cref{eq-main-counting-space}.
			 By \Cref{isometry-main-criterium}, the map $f: C \rightarrow \mathbb{F}_4^3$ is an $\mathbb{F}_2$-linear isometry. Moreover, by \Cref{extendibility-criterium}, since the tuples $(V_1, V_2, V_3)$ and $(U_1, U_2, U_3)$ are not equivalent, the isometry $f$ is unextendible.
		\end{example}
		
		Combining \Cref{extendibility-criterium} and \Cref{isometry-main-criterium}, we claim that a $K$-linear isometry is extendible if and only if the corresponding solution of \cref{eq-main-counting-space} is trivial. Nontrivial solutions of the equation must satisfy specific requirements on the subspace coverings. Such coverings and related questions are discussed in \cite{covnum} and are partially connected with our results. 
		
		\begin{lemma}\label{minimum-covering-number-of-one-space}
			Let $V$ be a nonzero vector space over $K$ and let $U_i \subset V$ be proper subspaces, for $i \in \myset{m}$. If $V = \bigcup_{i =1}^m U_i$, then $m$ is greater than the cardinality of ${K}$.
		\end{lemma}
		\begin{proof}
			For any $i \in \myset{m}$, $\dim_K U_i \leq \dim_K V - 1$ and hence $\card{U_i} \leq \frac{\card{V}}{\card{K}}$. Thus we have
			\begin{equation*}
				\card{V} < \sum_{i =1}^m \card{U_i} \leq m \frac{\card{V}}{\card{K}}
			\end{equation*} that implies $m > \card{K}$.
		\end{proof}

		\begin{lemma}\label{minimum-size-of-space-equation}
			Let $U_1,\dots, U_r, V_1, \dots, V_s$ be different spaces over $K$. Assume that $a_1,\dots, a_r, b_1,\dots, b_s > 0$ and
			\begin{equation*}
				\sum_{i = 1}^{r} a_i \id_{U_i} = \sum_{i = 1}^{s} b_i \id_{V_i} \,\, .
			\end{equation*}
			Then $\max\{r,s\}$ is greater than the cardinality of ${K}$.
		\end{lemma}

		\begin{proof}
			Among the spaces $V_1, \dots, V_s, U_1, \dots, U_r$ choose one that is maximal under inclusion. It is either $V_i$ for some $i \in \myset{s}$, or $U_j$ for some $j \in \myset{t}$. In the first case $V_i = \bigcup_{j = 1}^{r} (V_i \cap U_j)$, where for all $j \in \myset{r}$, $V_i \cap U_j \subset V_i$. From \Cref{minimum-covering-number-of-one-space}, $r > \card{K}$. Similarly, in the second case $s > \card{K}$.
		\end{proof}

		\begin{theorem}\label{theorem-less-then-m}
			Let $L$ be a finite field and let $K$ be a proper subfield of $L$. Let $m \leq \card{K}$ and let $C$ be a $K$-linear code in $L^m$. Any $K$-linear code isometry is extendible. Moreover, for any $m > \card{K}$ there exists a code in $L^m$ that has an unextendible $K$-linear isometry.
		\end{theorem}
		\begin{proof}
			Assume that there exist a $K$-linear code $C \subseteq L^m$ and an unextendible $K$-linear isometry $f : C \rightarrow L^m$. Let $(\U,\V)$ be two tuples of spaces that correspond to some basis of $C$ and the map $f$.
			Since $f$ is an isometry, \Cref{isometry-main-criterium} implies that \cref{eq-main-counting-space} holds and hence $(\U,\V)$ is a solution.
			By \Cref{extendibility-criterium}, the solution $(\U,\V)$ is nontrivial.
			Grouping equal terms on each side of \cref{eq-main-counting-space} we get,
			\begin{equation*}
			\sum_{i=1}^r a_i \id_{V'_i} = \sum_{i=1}^s b_i \id_{U'_i} \;,
			\end{equation*}
			where $V'_i$, $U'_j$ are $K$-linear spaces, $a_i,b_j >0$, for $i \in \myset{r}$, $j \in \myset{s}$, the spaces $V_i'$, for $i \in \myset{r}$, are all different and the spaces $U_i'$, for $i \in \myset{s}$ are all different.
			Note that $r,s \leq m$.
			Eliminate equal terms from different sides and make a renumbering of the spaces on both sides of the equation. The resulting equation is the following,
			\begin{equation*}
			\sum_{i=1}^{r'} a'_i \id_{V''_i} = \sum_{i=1}^{s'} b'_i \id_{U''_i} \;,
			\end{equation*}
			where $V''_i$, $U''_j$ are $K$-linear spaces, $a'_i,b'_j >0$, for $i \in \myset{r'}$, $j \in \myset{s'}$, and the spaces $V_i'', U_j''$, for $i \in \myset{r'}$, $i \in \myset{s'}$ are all different.
			In the last equation, all the conditions of \Cref{minimum-size-of-space-equation} are satisfied and therefore $\max\{r',s'\} > \card{K}$.
			Note that $r'\leq r\leq m$ and $s' \leq s\leq m$. Therefore $m> \card{K}$.
			
			For $m = \card{K} + 1$ we have already introduced a $K$-linear code in $L^m$ with unextendible $K$-linear isometry in \Cref{ex-1}. Evidently, for $m > \card{K} + 1$ such a pair of codes and an isometry is constructed by adding a set of arbitrary columns to the generator matrices of the two codes from \Cref{ex-1}.
		\end{proof}
		
		Of course, the techniques developed in the paper can be used to prove the classical MacWilliams Extension Theorem for linear codes. For the case $K = L$ we can refine \Cref{theorem-less-then-m}.
		
		\begin{proof}[Proof of the MacWilliams Extension Theorem]\label{proof-mac-willams}
			Due to \Cref{extendibility-criterium} and \Cref{isometry-main-criterium}, where the field $K$ is considered to be $L$, it is enough to show that all solutions $(\U,\V)$ of \cref{eq-main-counting-space} are trivial. By the definition of column space, for all $i \in \myset{m}$, $\dim_K V_i \leq n$ and $\dim_K U_i \leq n$, where $n = [L:K] = 1$. Therefore the spaces in $\U$ and $\V$ are just one-dimensional or zero spaces and hence a solution of \cref{eq-main-counting-space} can be only trivial.
		\end{proof}
		
		It is worth to note that, except the case $K = L$, in the paper we never used the fact that $L$ is a field. We only required $L$ to be a vector space over $K$. As we mentioned above, the character techniques allows the generalization of properties of codes over fields to the case of codes over rings and over modules. The generalization of \Cref{extendibility-criterium} and \Cref{isometry-main-criterium} to the case of codes linear over modules is possible and will appear in our further works.
		
	\section{Unextendible additive isometries}
		In this section we give a description of one family of nontrivial solutions of \cref{eq-main-counting-space} in the case of $m = q+1$, where by $q$ we denoted the cardinality of the field $K$. 
		As we mentioned above, nontrivial solutions of \cref{eq-main-counting-space} are in the correspondence with unextendible $K$-linear code isometries (see \Cref{extendibility-criterium} and \Cref{isometry-main-criterium}).
		
		In \Cref{minimum-covering-number-of-one-space} we proved that the covering of a space by proper subspaces is possible only if the number of subspaces is not less than $q + 1$. The following lemma gives the description all such possible covering.
		\begin{lemma}\label{minimum-space-covering-description}
			Let $V$ be a vector space over $K$ of dimension $k\geq 2$. Let $U_i, i \in \myset{q+1}$ be proper subspaces of $V$. If $V = \bigcup_{i = 1}^{q+1} U_i$, then there exists a subspace $S \subset V$ of dimension $k-2$ such that $\{U_1, \dots, U_{q+1}\}$ is the set of all subspaces of dimension $k-1$ that contain $S$.
		\end{lemma}

		\begin{proof}
			Assume that there are at least two spaces, let them be $U_q$ and $U_{q+1}$, with dimensions smaller than $k-1$. Then $q^k = \card{V} = \card{\bigcup_{i = 1}^{q+1} U_i} < \sum_{i = 1}^{q-1} \card{U_i} + \card{U_{q}} + \card{U_{q+1}} \leq (q - 1)q^{k -1} + 2q^{k -2} = q^k - q^{k-1} + 2q^{k-2}$, which is not true since $q \geq 2$. Therefore there exists at most one space $U_i$ with $\dim_K U_i \leq k-2$, $i \in \myset{q+1}$. Assume it exists and let it be $U_{q+1}$. For $i \in \{2,\dots,q+1\}$ define a set $\bar{U}_i = \bigcup_{j < i} U_j$ and notice that for $i \in \{2,\dots,q\}$, $\card{U_i \setminus \bar{U}_i} \leq \card{U_i} - \card{U_i \cap U_1} = q^{k-1} - q^{k - 2}$, because $\dim_K U_1 \cap U_i = k - 2$. The equality $\card{U_i \setminus \bar{U}_i} = q^{k-1} - q^{k - 2}$ holds if and only if $U_i \cap \bar{U}_i = U_i \cap U_1$, where $i \in \{2,\dots,q\}$. Obviously, $\card{U_{q+1} \setminus \bar{U}_{q+1}} \leq q^{k-2} - 1$.
			In the equality $V = U_1 \cup \bigcup_{i=2}^{q+1}(U_i \setminus \bar{U}_i)$ all sets in the union are disjoint. Thus
			\begin{equation*}
				q^k = \card{V} = \card{U_1} + \sum_{i =2}^{q+1} \card{U_i \setminus \bar{U}_i} \leq q^{k-1} + (q-1)(q^{k-1} - q^{k-2}) + q^{k-2} - 1\;.
			\end{equation*}
			Regrouping the terms we get $2q^{k-2} \geq q^{k-1} + 1$, which gives a contradiction. Hence, $\dim_K U_i = k -1$ for all $i \in \myset{q+1}$ and we can refine the inequality,
			\begin{equation*}
				q^k = \card{V} = \card{U_1} + \sum_{i =2}^{q+1} \card{U_i \setminus \bar{U}_i} \leq q^{k-1} + q(q^{k-1} - q^{k-2}) = q^k\;.
			\end{equation*}
			This implies that for all $i \in \{2, \dots, q+1\}$, $\card{U_i \setminus \bar{U}_i} = q^{k-1} - q^{k-2}$ and therefore $U_i \cap \bar{U}_i = U_i \cap U_1$.
			Define the space of dimension $k-2$, $S = U_2 \cap U_1$.
			The following equalities hold, $U_1 \cap U_i = U_i \cap \bar{U}_i = U_2 \cap (U_i \cap \bar{U}_i)  = U_2 \cap (U_i \cap U_1) = U_i \cap S$, for all $i \in \{2, \dots, q+1\}$. So $U_i \cap S$ has dimension $k-2$, which implies $S \subset U_i$, for all $i \in \{2,\dots, q+1\}$.
			Evidently, the spaces $U_i$ for $i \in \myset{q+1}$ include all the spaces that are strictly between $S$ and $V$.
		\end{proof}
		
		For a pair of spaces $S \subset V$ of dimensions $n - 2$ and $n$ correspondingly define two tuples of spaces $\U^A = (U^A_1, \dots, U^A_{q+1})$ and $\V^A = (V^A_1,\dots, V^A_{q+1})$ in the following way. Let $V^A_1 = \dots = V^A_q = V$, $V^A_{q+1} = S$ and let $U^A_1, \dots, U^A_{q+1}$ be all different hyperplanes in $V$ that contain $S$.
		
		\begin{proposition}\label{description-of-solutions-with-diff-dimensions}
		Let $\U$ and $\V$ be two tuples of spaces such that
		$$\max_{i \in \myset{q+1}} \dim_K V_i > \max_{i \in \myset{q+1}} \dim_K U_i\;.$$
		The pair $(\U, \V)$ is a nontrivial solution of \cref{eq-main-counting-space} if and only if there exist spaces $V$ and $S$ of dimension $k$ and $k-2$ correspondingly, such that $\U \sim \U^A$ and $\V \sim \V^A$.
		\end{proposition}
		\begin{proof}
			Prove the only if part.
			Without loss of generality, assume that $\dim_K V_1 = k = \max_{i \in \myset{q+1}} \dim_K V_i$. Obviously, $k \geq 2$ and
			from \cref{eq-main-counting-space}, $V_1 = \bigcup_{i=1}^{q+1} (U_i \cap V_1)$, where $U_i \cap V_1 \subset V_1$ for all $i \in \myset{q+1}$.
			From \Cref{minimum-space-covering-description} there exists a subspace $S \subset V_1$ such that $\dim_K S = k - 2$ and $S \subset U_i \cap V_1 \subset V_1$, $\dim_K U_i \cap V_1 = k - 1$ and all the spaces $U_i \cap V_1$ are different for $i \in \myset{q+1}$.
			From the conditions $\dim_K U_i \cap V_1 = k -1$, $\dim_K U_i < k$ and $U_i \cap V_1 \subset V_1$ we deduce $U_i = U_i \cap V_1 \subset V_1$, where $i \in \myset{q+1}$.
			Since $V_1 = \bigcup_{i=1}^{q+1} U_i$ it is easy to see that $\id_{V_1} + q \id_{S} = \sum_{i = 1}^{q+1} \id_{U_i}$. \Cref{eq-main-counting-space} can be rewritten as
			\begin{equation*}
				\frac{1}{q^k}\id_{V_1} + \sum_{i = 2}^{q+1} \frac{1}{\card{V_i}} \id_{V_i} =\frac{1}{q^{k-1}} \sum_{i = 1}^{q+1} \id_{U_i} = \frac{1}{q^{k-1}}\id_{V_1} + \frac{1}{q^{k-2}} \id_{S}\;.
			\end{equation*}
			Subtracting $q^{-k}\id_{V_1}$ from both sides we get
			\begin{equation*}
				\sum_{i = 2}^{q+1} \frac{1}{\card{V_i}} \id_{V_i} = \frac{q-1}{q^k} \id_{V_1} +  \frac{1}{q^{k-2}} \id_S \;.
			\end{equation*}
			Since $q > 1$, $V_1 = \bigcup_{i = 2}^{q+1} (V_i \cap V_1)$. From \Cref{minimum-space-covering-description}, considering the fact that the number of terms from both sides is less than $q+1$, there exists $i \in \{2,\dots,q+1\}$ such that $V_i = V_1$. Assume $V_2 = V_1$ and reduce the equation,
			\begin{equation*}
				\sum_{i = 3}^{q+1} \frac{1}{\card{V_i}} \id_{V_i} = \frac{q-2}{q^k} \id_{V_1} +  \frac{1}{q^{k-2}} \id_{S}\;.
			\end{equation*}
			Repeating the procedure $q - 2$ more times we get that $V_i = V_1$ for every $i \in \myset{q}$ and the reduced equation becomes $\frac{1}{\card{V_{q+1}}} \id_{V_{q+1}} =  \frac{1}{q^{k-2}} \id_{S}$. Obviously, $V_{q+1} = S$. Defining $V = V_1$, we proved that $\U \sim \U^A$ and $\V \sim \V^A$.
			
			In the other direction, easy to see that the pair $(\U^A, \V^A)$ is really a solution of \cref{eq-main-counting-space}.
		\end{proof}

		Having a family of nontrivial solutions for $m = q+1$ we can build a family of unextendible $K$-linear code isometries for codes of length $q+1$. The unextendible additive isometry presented in \Cref{counter-example-existence} is a particular case, which corresponds to the solution $(\U^A, \V^A)$ with $V = K^2$ and $S = \{0\}$.
		
		The full description of nontrivial solutions of \cref{eq-main-counting-space} will appear in the further paper.

		\footnotesize
		\bibliographystyle{ieeetr}
		\bibliography{biblio}
\end{document}